\documentclass[12pt]{amsart}
\usepackage{amsmath,amsfonts}
\usepackage{amssymb}
\usepackage{amscd}
\usepackage{amsthm}
\usepackage{mathtools} 
\usepackage{yhmath}

\usepackage{epsfig}

\usepackage{todonotes}

\usepackage[all]{xy}
\usepackage{color}

\numberwithin{equation}{section}

\newtheorem{thm}[equation]{Theorem}
\newtheorem{proposition}[equation]{Proposition}
\newtheorem{prop}[equation]{Proposition}

\newtheorem{lemma}[equation]{Lemma}

\newtheorem{corollary}[equation]{Corollary}

\theoremstyle{remark}
\newtheorem{remark}[equation]{Remark}

\theoremstyle{definition}
\newtheorem{definition}[equation]{Definition}

\def\XXint#1#2#3{{\setbox0=\hbox{$#1{#2#3}{\int}$}
	\vcenter{\hbox{$#2#3$}}\kern-.5\wd0}}

\newcommand{\ra}{\rightarrow}

\newcommand{\N}{\mathbb N}

\newcommand{\R}{\mathbb R}

\newcommand{\Z}{\mathbb Z}

\newcommand{\h}{\mathfrak{h}}

\newcommand{\Aut}{\operatorname{Aut}}

\newcommand{\id}{\operatorname{id}}

\newcommand{\Lie}{\operatorname{Lie}}

\newcommand{\Isome}{\operatorname{Isome}}
\newcommand{\Stab}{\operatorname{Stab}}
\newcommand{\nil}{\operatorname{nil}} 

\newcommand{\acts}{\curvearrowright}

\def\ra{\rightarrow}

\newcommand{\normal}{\vartriangleleft}

\newcommand{\dif}{\mathrm{d}}
\newcommand{\braket}[1]{\langle  #1 \rangle}
\newcommand{\bbraket}[1]{\langle\langle  #1 \rangle\rangle}

\newcommand{\clo}{\bar} 

\begin{document}

\title{Isometries of nilpotent metric groups}

\author[Kivioja]{Ville Kivioja}
\email{kivioja.ville@gmail.com}

\author[Le Donne]{Enrico Le Donne}
\email{enrico.ledonne@gmail.com}
\address[Le Donne and Kivioja]{
Department of Mathematics and Statistics, University of Jyv\"askyl\"a, 40014 Jyv\"askyl\"a, Finland}


\renewcommand{\subjclassname}{%
 \textup{2010} Mathematics Subject Classification}

\date{\today}

\begin{abstract}
We consider Lie groups equipped with arbitrary distances. We only assume that the distance is left-invariant and induces the manifold topology. For brevity, we call such object metric Lie groups.
Apart from Riemannian Lie groups, distinguished examples are sub-Riemannian Lie groups and, in particular, Carnot groups equipped with Carnot-Carath\'eodory distances.
We study the regularity of isometries, i.e., distance-preserving homeomorphisms.
Our first result is the analyticity of such maps between metric Lie groups.
The second result is that if two metric Lie groups are connected and nilpotent then every isometry between the groups is the composition of a left translation and an isomorphism.
There are counterexamples if one does not assume the groups to be either connected or nilpotent. 
The first result is based on a solution of the Hilbert 5th problem by Montgomery and Zippin.
The second result is proved, via the first result, considering the Riemannian case, which for self-isometries was solved by Wolf.
\end{abstract}

\maketitle
\tableofcontents

\section{Introduction}   			   


In this paper, with the term \emph{metric Lie group} we mean a Lie group
equipped with a left-invariant distance that induces the manifold topology.
An {\em isometry} is a distance-preserving bijection. Hence, a priori it is only a homeomorphism. 
As a general fact we  show 
the following regularity result.
\begin{thm}\label{smooth:thm}
Isometries between  metric Lie groups are analytic maps.
\end{thm}

We say that a map 
between groups is \emph{affine} if it is the composition of a left translation and a group homomorphism.
The main result of this paper is the following stronger result for nilpotent groups.
\begin{thm}\label{main:thm} 
Isometries between nilpotent connected metric Lie groups
are affine.
\end{thm}

In particular our result shows that 
\begin{itemize}
\item[{[\ref{main:thm}.i]}] two isometric nilpotent connected metric Lie groups are isomorphic.
\item[{[\ref{main:thm}.ii]}] the isometry group of a nilpotent connected metric Lie group is a semi-direct product, see \eqref{semidirect}.
\end{itemize}

Theorem \ref{main:thm} is a generalization of previous results. Indeed, on the one hand, 
Wolf (\cite{Wolf}, see also \cite{Wilson}) extended the classical fact that Euclidean isometries are affine to self-isometries of nilpotent Lie groups equipped with left-invariant \emph{Riemannian} distances.
On the other hand, Theorem \ref{main:thm} has been previously shown in the case of sub-Riemannian (and more generally sub-Finsler) Carnot groups, see  \cite{pansu,hamenstadt,
kishimoto,LeDonne-Ottazzi}.
In fact our strategy of proof is built on both \cite{Wolf,Wilson} and \cite{LeDonne-Ottazzi}.

We remark that both assumptions `connectedness' and `nilpotency'  are necessary for Theorem \ref{main:thm} to hold.
 In this respect in Section \ref{sec:counterexamples} we provide some counterexamples.

The large-scale analogue of Theorem \ref{main:thm} is a challenging open problem that raised a lot attention since the papers of Pansu and Shalom \cite{pansu,Shalom}. What is expected is that if two finitely generated nilpotent groups are torsion-free, then every quasi-isometry between them induces an isomorphism between their Malcev completions. Recent progress has been done in \cite{Kyed_Densing:Petersen}.

We spend the rest of the introduction to explain the strategy of the proofs of the two theorems and the structure of the paper.
To study isometries between two metric Lie groups, we first treat  the case when the two groups are the same, i.e., they are isometric via a  Lie group isomorphism. 
If \( M \) is a connected metric Lie group, we consider its isometry group \( G \), that is, the set of self-isometries of \( M \) equipped with the composition rule and the compact-open topology. 
Hence, the group \( G \)   acts continuously, transitively and by isometries on \( M \). 
It is crucial that \( G \) is a locally compact   group. 
This latter fact follows from 
Ascoli--Arzel\'a Theorem but it
needs some argument since closed balls are not necessarily assumed to be compact.
At this point, the theory of locally compact groups, \cite{mz}, provides  a Lie group structure  on \( G \) such that the action \( G \acts M \) is smooth.

Assume that $M_1, M_2$ are metric Lie groups and $F:M_1\to M_2$ is an isometry.
We consider the above-mentioned Lie group structures
on the respective isometry groups
$G_1$, $G_2$.
The conjugation by $F$ provides a map from 
$G_1$ to $G_2$ that is a continuous homomorphism between Lie groups, hence it is analytic.
This observation will give the conclusion of the proof of 
Theorem~\ref{smooth:thm}, see Section ~\ref{sec:smooth}.

An important consequence of 
Theorem~\ref{smooth:thm} is that every  isometry between metric Lie groups can be seen as a Riemannian isometry.
Namely, for every map 
$F:M_1\to M_2$ as above there are 
Riemannian left-invariant structures $g_1, g_2$ such that
$F:(M_1,g_1)\to (M_2,g_2)$
is a Riemannian isometry, see Proposition~\ref{prop:isometries_are_riemannian}.
Of a separate interest is the fact that the Riemannian structures can be chosen independently on $F$.



To prove Theorem \ref{main:thm}  we will use the above relationship with the Riemannian case together with 
  Wolf's study of    the nilradical of   isometry groups of nilpotent Riemannian Lie groups.
We recall that the nilradical of a Lie algebra \( \mathfrak{g} \), which we denote by \( \nil(\mathfrak{g}) \), is the biggest
nilpotent ideal of \( \mathfrak{g} \), see 
Section~\ref{sec:nilradical}.
By a theorem of 
Wolf (see Theorem \ref{thm_wilson}) if \( N \) is a 
nilpotent  Riemannian  Lie group, then 
\begin{equation}\label{nilradica_condition:intro} 
 \Lie(N) = \nil(\Lie(\Isome(N))) .
\end{equation} 
Here we are considering the natural inclusion \( N \hookrightarrow \Isome(N) \) via left translations and we are using  that $\Isome(N)$ is a Lie group.
One should think of \eqref{nilradica_condition:intro} as an
  algebraic characterization of a 
nilpotent Riemannian  Lie group  inside its isometry group.
We deduce that 
given an isometry 
\( F \colon N_1 \ra N_2 \) between such groups, the conjugation by $F$, which is an isomorphism between the isometry groups, induces an isomorphism between $N_1 $ and $ N_2$,  see 
Lemma \ref{lemma317}.
In fact, we show that 
if $F$ fixes the identity, then the above isomorphism coincides with
$F$, see 
Lemma \ref{Lemma316}. In general, we conclude that $F$ is an affine map.


For completeness, we  show that in the general context of 
 nilpotent connected metric Lie groups   the nilradical condition   \eqref{nilradica_condition:intro} still holds,
see Proposition \ref{star!}.
We also show that if
$M$ is a group equipped with a left-invariant distance, then
its
 isometries   are affine if and only if
its isometry group $G$ splits as semi-direct product
\begin{equation*}  
G = M \rtimes \mathrm{Stab}_1(G),
\end{equation*}
where \( \mathrm{Stab}_1(G) \) is the set of isometries fixing the identity element $1$ of $M$.
We provide the simple proof in Lemma \ref{TFAE}.

\subsubsection{Acknowledgement}
The authors thank
M.~Jablonski,
A.~Ottazzi, and
P.~Petersen for helpful remarks.
E.L.D. acknowledges the support of the Academy of Finland project no. 288501.


\section{Regularity of isometries}
\subsection{Lie group structure of isometry groups}

The first aim of this section is to show that the isometry group of a metric Lie group is a Lie group. Such a fact is a consequence of the solution of the Hilbert 5th problem by Montgomery--Zippin, together with the observation that the isometry group is locally compact. This latter property follows by Ascoli--Arzel\`a theorem.

We stress that a
metric Lie group $(M,d)$ 
may not be boundedly compact. Namely, the closed balls 
$\bar B_d(1_M, r):=
\{ p\in M : d(p,1_M)\leq r\}
$ with respect to $d$ may not be compact. For example, this is the case for the distance $\min\{d_E, 1\}$ on $\R$, where $d_E$ denotes the Euclidean distance.

\begin{remark} \label{rmk:boundedy_cpt}
If \( (M,d) \) is a connected metric Lie group, then there exists a distance \( \rho \) such that \( (M,\rho) \) is a metric Lie group that is boundedly compact and \( \Isome(M,d) \subseteq \Isome(M,\rho) \).
Indeed, since the topology induced by $d$ is the manifold topology, then there exists some $r_0>0$ such that 
$\bar B_d(1_M, r_0)$ is compact.
Then we can consider the distance
$$\rho(p,q):=
\inf
\{ 
\sum_{i=1}^k d(p_{i-1}, p_i) \,:\, k\in \N,\, p_i\in M,\, p_0=p, \,p_k=q,\, d(p_{i-1}, p_i)\leq r_0
 \}.$$
 Once can 
 check that $(M,\rho)$ is a metric Lie group, 
 for all $r>0$ the set
 $\bar B_\rho(1_M, r)$  is compact, and
 $\Isome(M,d) \subseteq \Isome(M,\rho).$
\end{remark}

Let us clarify now why the isometry group of a connected metric Lie group is locally compact, which was not justified in \cite{LeDonne-Ottazzi}. With the terminology of Remark \ref{rmk:boundedy_cpt} the stabilizer \( S \) of \( 1 \) in \( \Isome(M,d) \) is a closed subgroup of the stabilizer \( S_\rho \) of \( 1 \) in \( \Isome(M,\rho) \). Furthermore, for any \( r > 0 \) and \( f \in S_\rho \) we have that \( f(\clo{B}_\rho(1,r)) = \clo{B}_\rho(1,r) \), which is compact. Hence, the maps from \( S \) restricted to \( \clo{B}(1,r) \) form an equi-uniformly continuous and pointwise precompact family. Ascoli--Arzel\`a Theorem implies that \( S_\rho \) is compact, being also closed in \( C^0(M,M) \). Consequently, \( S \) is compact and because \( M \) is locally compact, then also \( \Isome(G,d) \) is locally compact. At this point we are allowed to use the theory of locally compact groups after Gleason--Montgomery--Zippin \cite{mz}. 
In fact, the argument in \cite[Proposition~4.5]{LeDonne-Ottazzi} concludes the proof of the following result.
\begin{proposition}\label{Lie_structure}
Let  \(M \) be  a metric Lie group with isometry group $G$. Assume that $M$ is connected.
\begin{enumerate}
\item The stabilizers of the action \( G \acts M \) are compact.
\item The topological group $G$ is a Lie group (finite dimensional and with finitely many connected components) acting analytically on $M$.
\end{enumerate}
\end{proposition}

\begin{remark}
The assumption of $M$ being connected in Proposition \ref{Lie_structure} is necessary. Indeed, one can take as a counterexample the group $\Z$ with the discrete distance.
\end{remark}

\subsection{Proof of smoothness} \label{sec:smooth}
With the use of
Proposition \ref{Lie_structure}, we  give the proof of the analyticity of isometries (Theorem \ref{smooth:thm}).
We remark that in the Riemannian setting the classical result of 
Myers and Steenrod gives smoothness of isometries, see
\cite{Myers-Steenrod}, and more generally
\cite{Capogna_LeDonne}.
However,  the following proof is different in spirit and, nonetheless, it will imply (see Proposition \ref{prop:isometries_are_riemannian}) that such metric isometries are Riemannian isometries for some Riemannian structures. 

\begin{proof}[Proof of Theorem \ref{smooth:thm}] 
Let \( F \colon M_1 \to M_2 \) be an isometry between metric Lie groups.
Without loss of generality we may assume that \( p=1_{M_1} \), \( F(1_{M_1})=1_{M_2} \) and that both \( M_1 \) and \( M_2 \) are connected, since left translations are analytic isometries and connected components of identity elements are open. By Proposition \ref{Lie_structure}, for \( i \in \{ 1,2 \} \), the space \( G_i \coloneqq \Isome(M_i) \) is a Lie group. The map \( \mathrm{C}_F \colon G_1 \to G_2 \) defined as \( I \mapsto F \circ I \circ F^{-1} \) is a group isomorphism that is continuous, see \cite[Theorem 4]{Arens}. Hence, the map \( \mathrm{C}_F \) is analytic, see \cite[p.~117 Theorem~2.6]{Helgason}.

Consider also the inclusion \( \iota \colon M_1 \to G_1 \), \( m \mapsto L_m \), which is analytic being a continuous homomorphism, and the orbit map \( \sigma \colon G_2 \to M_2 \), \( I \mapsto I(1_{M_2}) \), which is analytic since the action is analytic (Proposition \ref{Lie_structure}).
We deduce that \( \sigma \circ \mathrm{C}_F \circ \iota \) is analytic.
We claim that this map is \( F \). Indeed,
for any \( m \in M_1 \) it holds
\[
(\sigma \circ \mathrm{C}_F \circ \iota)(m) 
=\sigma(F \circ L_m \circ F^{-1})
=(F \circ L_m \circ F^{-1})(1_{M_2})
=F(m).
\qedhere
\]
\end{proof} 

\subsection{Isometries as Riemannian isometries}

We show next that isometries between metric Lie groups are actually Riemannian isometries for some left-invariant structures.
Let us point out that when $M$ is a Lie group and $g$ is a  left-invariant Riemannian metric tensor on $g$, then one has an induced Riemannian distance $d_g$ and, by the theorem of Myers and Steenrod  
\cite{Myers-Steenrod},   the group $\Isome(M,d_g) $  of  distance-preserving bijections
 coincides with the group $\Isome(M,g)$ of   tensor-preserving diffeomorphisms.
 In what follows we shall write $(M,g)$ to denote the metric Lie group $(M,d_g) $.

\begin{proposition} \label{prop:isometries_are_riemannian}
If \( (M_1,d_1) \) and \( (M_2,d_2) \) are connected metric Lie groups, then
there exists left-invariant Riemannian metrics \( g_1 \) and \( g_2 \) on \( M_1 \) and \( M_2 \), respectively, such that 
\( \Isome(M_i,d_i) \subseteq \Isome(M_i,g_i) \) for \( i \in \{1,2\} \) and for all isometries
\( F \colon (M_1,d_1) \ra (M_2,d_2) \) the map 
\( F \colon (M_1,g_1) \ra (M_2,g_2) \) is a Riemannian isometry.
\end{proposition}
Let us first deal with the case \( (M_1,d_1) = (M_2,d_2) \).
\begin{lemma}\label{lemma:riem_metric}
If \( (M,d) \) is  a connected  metric Lie group, then there is a Riemannian metric \( g \) such that 
\(\Isome(M,d) \subseteq \Isome(M,g) \).
\end{lemma}

\begin{proof}[Proof of Lemma \ref{lemma:riem_metric}] 
Fix a scalar product 
\( \bbraket{\cdot,\cdot} \) 
on the tangent space \( T_1 M \) at the identity 1 of $M$. From Proposition \ref{Lie_structure}, the stabilizer \( S \) of \( 1 \) in \( \Isome(M,d) \) is compact and acts smoothly on \( M \). Let \( \mu_S \) be the probability Haar measure on \( S \). Consider for \( v,w \in T_1 M \) 
\[ \braket{v,w} \coloneqq \int_S \bbraket{\dif F v, \dif F w} \dif \mu_S(F) .\]
Then \( \braket{\cdot,\cdot} \) defines a \( S \)-invariant scalar product on \( T_1 M \), and one can take \( g \) as the left-invariant Riemannian metric that coincides with \( \braket{\cdot,\cdot} \) at the identity.
\end{proof} 
\begin{proof}[Proof of Proposition \ref{prop:isometries_are_riemannian}] 
By Lemma \ref{lemma:riem_metric}
let \( g_2 \)   be a Riemannian metric on \( M_2 \) with     \begin{equation}\label{eq:incl2}
 \Isome(M_2,d_2) \subseteq \Isome(M_2,g_2) .\end{equation} 
  Fix \( F \colon (M_1,d_1) \ra (M_2,d_2) \) an isometry. By Theorem~\ref{smooth:thm} the map \( F \) is smooth, and we may define a Riemannian metric on \( M_1 \) by \( g_1 \coloneqq F^* g_2 \). There are two things to check: a) \( \Isome(M_1,d_1) \subseteq \Isome(M_1,g_1) \), which in particular gives that \( g_1 \) is left-invariant and b) every isometry \( H \colon (M_1,d_1) \ra (M_2,d_2) \) is an isometry of Riemannian manifolds.

For part a, since by construction \( F \) is also a Riemannian isometry, the map \( I \mapsto F \circ I \circ F^{-1} \) is a bijection between \( \Isome(M_1,d_1) \) and \( \Isome(M_2,d_2) \) and between \( \Isome(M_1,g_1) \) and \( \Isome(M_2,g_2) \).
Therefore the inclusion
\eqref{eq:incl2} implies the inclusion
\( \Isome(M_1,d_1) \subseteq \Isome(M_1,g_1) \).

For part b, 
since \( H \circ F^{-1} \in \Isome(M_2,d_2) \subseteq \Isome(M_2,g_2) \), then \( (H \circ F^{-1})^* g_2 = g_2 \). Consequently, we get \( H^* g_2 = F^* (H \circ F^{-1})^* g_2 = g_1 \).
\end{proof} 

\section{Affine decomposition}
\subsection{Preliminary lemmas}

Given a  group $M$ we denote   by $M^L$ the group of left translations on $M$.
The following two results make sense in the settings of groups equipped with  left-invariant distances. We call such groups {\em metric groups}.
\begin{lemma}\label{Lemma316}
Let $M_1$ and $M_2$ be metric groups. 
Suppose $ F : M_1 \to M_2$
is an isometry   and 
$ F\circ M_1^L \circ F^{-1} = M^L_2 $. 
Then F is affine.
\end{lemma}
\begin{proof} 
Up to precomposing with a translation, we assume that
 $F(1_{M_1})=1_{M_2}$. So we want to prove that \(F \) is an isomorphism.
The map \( \mathrm{C}_F \colon \Isome(M_1) \to \Isome(M_2) \), \( I \mapsto F \circ I \circ F^{-1} \), is an isomorphism and by assumption it gives an isomorphism between \( M_1^L \) and \( M_2^L \). We claim that \( F \) is the same isomorphism when identifying \( M_i \) with \( M_i^L \). Namely, we want to show that for all \( m \in M_1 \) we have \( L_{F(m)} = \mathrm{C}_F(L_m) \). By assumption for all \( m_1 \in M_1 \) exists \( m_2 \in M_2 \) such that \( L_{m_2} = \mathrm{C}_F(L_{m_1}) \). Evaluating at \( 1_{M_2} \), we get
\[
m_2 = 
L_{m_2}(1_{M_2}) = \mathrm{C}_F(L_{m_1})(1_{M_2}) = 
F( L_{m_1}(F^{-1}(1_{M_2}) )) 
= F(m_1).
\qedhere
\]
\end{proof} 

With the next result we clarify that the condition of self-isometries being affine is equivalent to left translations being a normal subgroup of the group of isometries. Equivalently, we have a semi-direct product decomposition of the isometry group. 
Namely, given a metric group $M$ and denoting by
$G$ the isometry group
and by 
$\mathrm{Stab}_1 (G) $ the stabilizer of the identity element, we have that 
  $M$ has affine isometries if and only if 
 \( G = M^L \rtimes \mathrm{Stab}_1 (G) \).
We denote by $\mathrm{Aff}(M)$ the group of affine maps from  $M$ to $M$ and by $\mathrm{Aut}(M) $ the group of automorphisms of $M$.
\begin{lemma}\label{TFAE}
Let \( M \) be a metric group with isometry group \( G \).
Then the following are equivalent:
\begin{enumerate}
\renewcommand{\theenumi}{\alph{enumi}}
\item \( M^L \normal G \), 
i.e.,  $ F\circ M^L \circ F^{-1} = M^L$, for all $F\in G$;

\item
$G< \mathrm{Aff}(M)$;

\item
 \( \mathrm{Stab}_1 (G) < \mathrm{Aut}(M) \);

\item
 \( G = M^L \rtimes \mathrm{Stab}_1 (G) \).
\end{enumerate}
\end{lemma}
\begin{proof} 
Property (a) implies (b) by  Lemma \ref{Lemma316}.
Regarding the fact that (b) implies (a),  consider a map \( F \in G \), which we know is of the form \( F = \tau \circ \Phi \) with \( \tau \in M^L \) and \( \Phi \in \Aut(M) \). For all \( p \in M \) we get
\begin{eqnarray*}
F \circ L_p \circ F^{-1}
&=& (\tau \circ \Phi) \circ L_p \circ(\tau \circ \Phi)^{-1}
\\
&=& \tau \circ \Phi \circ L_p \circ \Phi^{-1} \circ \tau^{-1}
\\
&=& \tau \circ L_{\Phi(p)} \circ \Phi \circ \Phi^{-1} \circ \tau^{-1}
\\
&=& \tau \circ L_{\Phi(p)} \circ \tau^{-1} \in M^L,
\end{eqnarray*}
which gives \( M^L \normal G \).


The equivalence of (b)  with (c) is trivial.
The equivalence of (a)  with (d) 
follows from the facts
 $M^L \cdot  \mathrm{Stab}_1 (G) = G$ and $M^L \cap  \mathrm{Stab}_1 (G) =\{\id\}$. 
\end{proof} 

\subsection{The nilradical condition}\label{sec:nilradical}


\begin{definition}[Nilradical condition] \label{def:nilrad}
Let \( \mathfrak{g} \) be a Lie algebra. The \emph{nilradical} of 
\( \mathfrak{g} \), denoted by \( \nil(\mathfrak{g}) \), is the largest nilpotent ideal of \( \mathfrak{g} \).
We say that a connected metric Lie group \( N \)
with isometry group $\Isome(N)$
 satisfies \emph{the nilradical condition} if it holds 
\begin{equation}\label{nilradica_condition} 
 \Lie(N^L) = \nil(\Lie(\Isome(N))) .
\end{equation} 
\end{definition}
Clearly, a metric Lie group \( N \) 
 satisfies  the nilradical condition only if it is nilpotent.
 The nilradical of a Lie algebra \( \mathfrak{g} \) can   also be defined as  the sum of all nilpotent ideals of
\( \mathfrak{g} \), see  \cite[Definition 5.2.10]{Hilgert_Neeb:book}.

The nilradical condition is introduced after 
\cite{Wolf, Wilson, Gordon-Wilson}. The importance of such condition is shown by the following  result.

\begin{lemma} \label{lemma317}
Let 
 $N_1$ and $N_2$ 
 be two connected  metric Lie groups. 
Assume that both $N_1$ and $N_2$ satisfy the nilradical condition \eqref{nilradica_condition}.
Then any isometry \( F \colon N_1 \ra N_2 \) is affine.
\end{lemma}
\begin{proof} 
As before, denote by \( \mathrm{C}_F \colon \Isome(N_1) \to \Isome(N_2) \) the map \( \mathrm{C}_F(I) = F \circ I \circ F^{-1} \). By Lemma \ref{Lemma316} it is enough to show that \( \mathrm{C}_F(N_1^L) = N_2^L \).
Recall by Proposition~\ref{Lie_structure} that for \( i \in \{ 1,2 \} \) the space  \( G_i \coloneqq \Isome(N_i) \) is a Lie group. As pointed out in the proof of Theorem \ref{smooth:thm}, the map \( \mathrm{C}_F \) is a smooth isomorphism. Let us consider the map at the level of Lie algebras, denoting by \( \mathfrak{g}_i \) and \( \mathfrak{n}_i \) the Lie algebras of \( G_i \) and \( N_i^L \), respectively. Namely, we consider the differential
\(
\dif \mathrm{C}_F \colon \mathfrak{g}_1 \to \mathfrak{g}_2
\),
which is an isomorphism of Lie algebras. Hence \( \dif \mathrm{C}_F(\nil(\mathfrak{g}_1)) = \nil(\mathfrak{g}_2) \).
Since by the assumption \( \nil(\mathfrak{g}_i) = \mathfrak{n}_i \), we get \( \dif \mathrm{C}_F(\mathfrak{n}_1 ) = \mathfrak{n}_2 \). Recall the formula \( \mathrm{exp} \circ \dif \mathrm{C}_F = \mathrm{C}_F \circ \mathrm{exp} \) and the fact that for connected nilpotent Lie groups the exponential map is surjective, see \cite{Warner,Corwin-Greenleaf}. Then we deduce that
\[
\mathrm{C}_F(N_1^L) = \mathrm{C}_F( \exp(\mathfrak{n}_1)) = \exp( \dif \mathrm{C}_F( \mathfrak{n}_1)) = \exp( \mathfrak{n}_2) = N_2^L
.\]
The claim follows by Lemma \ref{Lemma316}.
\end{proof} 

The  nilradical condition is satisfied by Riemannian nilpotent Lie groups, where the distance is induced by a left-invariant metric tensor. Such a result was proved by Wolf
 \cite[p.~278 Theorem 4.2]{Wolf}, see also \cite[p.~341 Theorem 2]{Wilson}.
\begin{thm}[{Wolf}]
\label{thm_wilson}
Every Riemannian 
nilpotent connected Lie group
satisfies
the  nilradical condition \eqref{nilradica_condition}
\end{thm}

We remark that the work of Wolf, together with the work of Gordon and Wilson, is one of the initial steps in the study of  (Riemannian)
nilmanifolds,
solvmanifolds, and homogeneous  Ricci solitons,
see \cite{Gordon-Wilson, Jablonski_HRS,
Jablonski_SSS}.

\subsection{Proof of affinity and its consequences} \label{sec:affinity}

Let \( (N_1,d_1) \) and \( (N_2, d_2) \) be two nilpotent connected metric Lie groups. 
Let \( F \colon (N_1,d_1) \ra (N_2 ,d_2)\) 
 be an isometry. We shall prove that $F$
 is affine.
By Proposition \ref{prop:isometries_are_riemannian} for \( i \in \{ 1,2 \} \) there exist metric tensors \( g_i \) on \( N_i \) such that
\( F \colon (N_1,g_1) \ra (N_2 ,g_2)\) is an isometry. By Wolf's theorem (Theorem \ref{thm_wilson}) the space \( (N_i,g_i) \) satisfies the nilradical condition \eqref{nilradica_condition}. Hence, we apply Lemma \ref{lemma317} with \( N_i = (N_i,g_i) \) and conclude that \( F \) is affine.

In particular, we have that \( N_1 \) and \( N_2 \) are isomorphic as Lie groups. Because of Lemma \ref{TFAE} we also deduce that isometry group of a nilpotent connected metric Lie group \( N \) is the semi-direct product
\begin{equation}\label{semidirect}
\Isome(N) = N \rtimes \operatorname{AutIsome}(N),
\end{equation}
where \( \operatorname{AutIsome}(N)\) denotes the group of automorphisms of \( N \) that are also isometries of \( N \).
\qed

\subsection{The nilradical of nilpotent Lie groups}

In this section we show that,
likewise in Theorem \ref{thm_wilson}, the 
nilradical of the isometry group of
 a nilpotent connected metric Lie group
 is the group of left translations. 
Namely, 
 the  nilradical condition \eqref{nilradica_condition} holds.
We start with a basic lemma and for the reader convenience we provide its proof.

\begin{lemma} 
\label{korvike-max-cpt}
Let \( N \) be a connected nilpotent Lie group. 
Then every compact subgroup \( K < N \) is central in \( N \).
\end{lemma}

\begin{proof} 
Let the Lie group \( \tilde N \) be the universal covering space of the Lie group \( N \)
in such a way that the projection map \( \pi \colon \tilde N \ra N \) is 
a morphism of Lie groups. 
Consider the sets \( B \coloneqq \exp^{-1} (\pi^{-1}(K)) \), \( A \coloneqq \exp^{-1} (\pi^{-1}(e)) \), and
 \( A_S \coloneqq \mathrm{span}_{\mathbb{R}}(A) \).
Since \( \tilde N \) is connected and \( \pi^{-1}(e) \) is discrete and normal in \( \tilde N \), the subgroup \( \pi^{-1}(e) \) is central in \( \tilde N \). Since \( \tilde N \) is nilpotent, we have (see \cite[p.~24]{Corwin-Greenleaf})
\begin{equation}  \label{eq:GG}
\exp(Z(\Lie(\tilde N))) = Z(\tilde N).
\end{equation}
Thus, since the exponential map is invertible, we deduce that \( A \) is central in \( \Lie(\tilde N) \). Hence \( \exp(b+A) = \exp(b) \exp(A) \) for all \( b \in B \). So the map \( \mathrm{exp}|_B \) induces a homeomorphism between \( B/A \) and \( \exp(B)/\exp(A) = \pi^{-1}(K)/\pi^{-1}(e) \cong K \). Thus the space \( B/A \) is compact. Since the canonical map from \( B/A \) to \( B/A_S \) is a continuous surjection, also \( B/A_S \) is compact.

We claim that \( B \subseteq A_S \). Indeed, by contradiction assume that there is \( b \in B \backslash A_S \). Notice that for all \( n \in \mathbb{N} \) we have \( nb \in B \) since \( \exp(nb) =\exp(b)^n \in \pi^{-1}(K) \). Since we proved that \( B / A_S \) is compact, the sequence \( nb + A_S \in B/A_S\) accumulates to a point \( x \in A/B_S \).
Consider a Euclidean distance \( d_E \) on \( \Lie(\tilde N) \). On the one hand, since \( b \) is not in the subspace \( A_S \), we have \( d_E(nb, A_S) \ra \infty \) as \( n \ra \infty \). On the other hand,
\[
d_E(nb,A_S) \le d_E(nb,x + A_S) + d_E(x+A_S,A_S)
\]
accumulates to some finite value. We deduce that
\[ K = \pi(\exp(B)) \subseteq \pi(\exp(A_S)) \subseteq \pi(\exp(Z(\Lie(\tilde N)))) = \pi(Z(\tilde N)) \subseteq Z(N)
\]
where we used   the fact that \( A_S \) is central and the equation \eqref{eq:GG}.
\end{proof} 

\begin{prop}
\label{star!} 
For a nilpotent connected metric Lie group \( N \) the  nilradical condition \eqref{nilradica_condition} holds.
\end{prop}

\begin{proof}
Let $  \mathfrak{  n} \coloneqq \Lie(N^L)$,
$G \coloneqq \Isome(N)
$,
   $  \mathfrak{\h}  \coloneqq  \nil(\Lie(G))$, and \( H := {\exp(  \mathfrak{h} )} \).
   Notice that $H$ is  a Lie subgroup  since on nilpotent Lie groups the exponential map is surjective.

From Theorem \ref{main:thm} and Lemma \ref{TFAE} we have \( N^L \normal G \). Therefore, 
\( \mathfrak{n} \) is a nilpotent ideal of \( \mathfrak{g} \). 
Thus \( \mathfrak{n} \subseteq \mathfrak{h} \).
Alternatively, 
one can use
 Proposition \ref{prop:isometries_are_riemannian} to get a Riemannian metric  \( g \) with \( G < \Isome(N,g) \). Then Theorem \ref{thm_wilson}
 implies that \( \mathfrak{n} \) is an   ideal of     \( \Lie(\Isome(N,g)) \) and hence it is an   ideal   of \( \Lie(G) \).

Regarding the other inclusion, take \( \phi \in H \). We need to show that \( \phi \) is a left translation. Since  \( \mathfrak{n} \subseteq \mathfrak{h} \) and so \( N^L \subseteq H \), we may assume \( \phi(1)=1 \), i.e., \( \phi \in \Stab_1(G) \). Since \( \Stab_1(G) \cap \clo{H} \), which is a compact subgroup of the connected nilpotent group \( \clo{H} \) (the topological closure of \( H \) in \( G \)), is central in \( \clo{H} \) (see Lemma \ref{korvike-max-cpt}), we get that
   \( \phi \in \Stab_1(G) \cap H \subseteq Z(\clo{H}) \). Since \( N^L \subseteq \clo{H} \), we get for all \( q \in N \)
  \[ \phi(q) = (\phi \circ L_q)(1) = (L_q \circ \phi)(1) = L_q(1) = q.  \]
  Thus \( \phi = \mathrm{id} \) and so \( H \subseteq N \) and \( \mathfrak{h} \subseteq \mathfrak{n} \).
\end{proof} 


\section{Examples for the sharpness of the assumptions}
\label{sec:counterexamples}
In this section we provide several examples to illustrate the sharpness of the assumptions in Theorem \ref{main:thm}.
Namely, we show that 
if one of the groups is not assumed
connected and nilpotent then there may be isometries that are not affine.

Regarding the connectedness assumption,
there are examples of Abelian metric Lie groups with finitely many components for which some isometries are not affine. One of the simplest examples is the subgroup of \( \mathbb{C} \) consisting of the four points \( \{1,\mathrm{i},-1,-\mathrm{i}\} \) equipped with the discrete distance. Here every permutation is an isometry. However, any automorphism needs to fix \( -1 \), since it is the only point of order 2.

Regarding the nilpotent assumption, there are both compact and non-compact examples. We remark that in any group equipped with a bi-invariant distance the involution is an isometry. Consequently, every compact group admits a distance for which the involution is an isometry. Such a map is a group isomorphism only if the group is Abelian. Nonetheless, we point out  the following fact which is a consequence of the work of Baum--Browder and Ochiai--Takahashi, see \cite{Baum_Browder,  Ochiai_Takahashi} and also \cite{Scheerer, Hubbuck_Kane}.

\begin{corollary} \label{coroll:scheerer}
Let \( G_1, G_2 \) be connected compact simple metric Lie groups. If \( F \colon G_1 \to G_2 \) is an isometry,
then \( G_1\) and \( G_2 \) are isomorphic as Lie groups.
If, moreover, 
\( G_1, G_2 \) are the same metric Lie group and
 $F$ is homotopic to the identity map via isometries,
 then \( F \) is affine.
\end{corollary}

%

We point out that there exist examples of pairs of metric Lie groups that are isomorphic as Lie groups and are isometric, but are not isomorphic as metric Lie groups: an example is the rototranslation group (see below) with different Euclidean distances.

Other interesting results for isometries between compact groups can be found in \cite{Ozeki} and \cite{Gordon80}.

The conclusion of Corollary \ref{coroll:scheerer} may not hold for arbitrary connected metric Lie groups. In fact, we recall the following example, due to Milnor \cite[Corollary~4.8]{Milnor}, of a group that is solvable and isometric to the Euclidean 3-space. Let \( G \) be the universal cover of the group of orientation-preserving isometries of the Euclidean plane, which is also called the rototranslation group. Such a group admits coordinates making it diffeomorphic to \( \mathbb{R}^3 \) with the product
\[
\left [ \begin{matrix} x \\ y \\ z \end{matrix}  \right]
\cdot
\left [ \begin{matrix} x' \\ y' \\ z' \end{matrix}  \right]
=
\left [ \begin{matrix} \cos z & - \sin z & 0 \\ \sin z & \cos z & 0 \\ 0 & 0 & 1 \end{matrix}  \right]
\left [ \begin{matrix} x' \\ y' \\ z' \end{matrix}  \right]
+
\left [  \begin{matrix} x \\ y \\ z \end{matrix}  \right]
.
\]
In these coordinates, the Euclidean metric is left-invariant. On the one hand, one can check that the isometries that are also automorphisms of \( G \) form a 1-dimensional space. On the other hand, the isometries fixing the identity element and homotopic to the identity map form a group isomorphic to \( \mathrm{SO}(3) \). Hence, we conclude that not all such isometries are affine. 
Moreover, this group gives an example of a non-nilpotent metric Lie group isometric (but not isomorphic) to a nilpotent connected metric Lie group, namely the Euclidean 3-space.

Notice that also the Riemannian metric with orthonormal frame $\partial_x, \partial_y, 2\partial_z$ gives a left-invariant structure on $G$, which is isometric to the previous one, but there is no isometric automorphism between the two structures.
Hence, these spaces are not isomorphic as metric Lie groups.

A further study of metric Lie groups isometric to nilpotent metric Lie groups can be found in \cite{Cowling_LeDonne_Ottazzi}. In particular, all such groups are solvable.

We finally recall another example.
The unit disc in the plane admits a group structure that makes the
hyperbolic distance left-invariant. In this metric Lie group not all isometries are affine.


\bibliography{general_bibliography}
\bibliographystyle{amsalpha}
\end{document}